\documentclass[11pt,reqno]{amsart}

\usepackage[shortlabels]{enumitem}
\usepackage{esint}
\usepackage{physics}
\usepackage{amssymb}
\usepackage{amsthm}
\usepackage{relsize}
\usepackage{xcolor}
\usepackage{tikz-cd} 
\usepackage[backref]{hyperref}
\usepackage{mathtools}
\usepackage{amsfonts}
\usepackage[all]{xy}
\usepackage{geometry}
\usepackage{mathrsfs}
\newtheorem{theorem}{Theorem}[section]
\newtheorem{lemma}[theorem]{Lemma}

\newtheorem{corollary}[theorem]{Corollary}
\newtheorem{conjecture}{Conjecture}
\newtheorem{question}{Question}

\theoremstyle{remark}
\newtheorem{remark}[theorem]{Remark}
\numberwithin{equation}{section}

\newcommand{\ii}{\ensuremath{\sqrt{-1}}}
\newcommand{\pp}{\bar\partial}
\newcommand{\vol}{\mathrm{Vol}}
\everymath{\displaystyle}

\begin{document}

\title[On the local topology of Ricci limit spaces]{On the local topology of non-collapsed Ricci bounded limit spaces}
\author{Song Sun}
\address{Institute for Advanced Study in Mathematics, Zhejiang University, Hangzhou 310058, China}
\email{songsun@zju.edu.cn}
\author{Jikang Wang}
\address{Department of Mathematics, University of California, Berkeley, CA 94720, USA}
\email{jikangwang1117@gmail.com}
\author{Junsheng Zhang}
\address{Courant Institute of Mathematical Sciences\\
  New York University, 251 Mercer St\\
  New York, NY 10012\\}
\curraddr{}
\email{jz7561@nyu.edu}
\thanks{}


\begin{abstract}We show that for a pointed Gromov-Hausdorff limit of non-collapsed Riemannian manifolds with bounded Ricci curvature, the local $b_1$
  of the regular loci vanishes. We also discuss applications and some open questions.

\end{abstract}

\maketitle
\section{Introduction}
A \emph{Ricci limit space} is, by definition, a pointed Gromov--Hausdorff limit of complete Riemannian manifolds whose Ricci curvature is uniformly bounded from below.  It is known that a Ricci limit space is always semi-locally simply connected (see \cite{Pan-Wei,Wang2022,Wang2021}). 
When an additional upper bound on the Ricci curvature and the volume non-collapsing condition are imposed, the regular set is an open subset with full top-dimensional Hausdorff measure and admits a smooth manifold structure \cite{CC97}.

In this paper, we are interested in studying the local topology of the regular loci around a singular point.   Before stating the main result, we introduce some notations.
For \( v \in (0,1] \) and \( n \in \mathbb{Z}_+ \), let \( \mathcal{M}(n, v) \) denote the set of pointed Gromov--Hausdorff limits of pointed complete Riemannian manifolds \( (M, g, p) \) of dimension \( n \) satisfying 
\begin{equation}
    |\mathrm{Ric}(g)| \leq  n-1
\end{equation} and 
    \begin{equation}
         \begin{aligned}
         \underline{\vol}(B_r(q))&:=\frac{\vol(B_r(q))}{\omega_nr^n} \geq v, \text{ for any $r\in (0,1]$ and $q\in B_1(p)$.}
    \end{aligned}
    \end{equation}
    Notice that if $(X, d, p)\in \mathcal M(n, v)$, then $(X, \lambda d, p)\in \mathcal M(n, v)$ for all $\lambda\geq 1$. 

    Denote by $B_r^n(0)$ the standard ball of radius $r$ in $\mathbb R^n$. Following \cite{CC97}, given \( (X, d, p) \in \mathcal{M}(n, v) \), for  \( \epsilon, \delta > 0 \), we denote the ($\epsilon, \delta$)-\emph{effective regular set} by 
$$\mathcal{R}_{\epsilon, \delta}  = \left\{ x \in X \mid d_{GH}\left( B_r(x), B_r^n(0) \right) \leq \epsilon r, \text{ for all } 0 < r \leq \delta \right\}.$$
We define the $\epsilon$-\emph{regular set} and the \emph{regular set} by
$$\mathcal{R}_\epsilon  = \bigcup_{\delta > 0} \mathcal{R}_{\epsilon, \delta},\quad  \mathcal{R}  =\bigcap_{\epsilon > 0} \mathcal R_\epsilon= \bigcap_{\epsilon > 0} \bigcup_{\delta > 0} \mathcal{R}_{\epsilon, \delta}. $$
Denote by $\mathcal S=X\setminus \mathcal R$ the \emph{singular set}. Sometimes to emphasize the space $X$ itself, we also write $X^{reg}=\mathcal{R}$ and $X^{sing}=\mathcal{S}$. 

The main result of this paper is the following:

\begin{theorem}[Exponent bound of local homology group]\label{main thm-unrescaled version}
    Given \( v \in (0,1] \) and \( n \in \mathbb{Z}_+ \), there exist constants \( t= t(n, v) > 0 \) and \( C = C(n, v) > 0 \) such that for any \( (X, d, p) \in \mathcal{M}(n, v) \), the local homology group
    \[
    \mathrm{Im}\left( H_1(B_t(p) \cap \mathcal{R}, \mathbb{Z}) \rightarrow H_1(B_1(p) \cap \mathcal{R}, \mathbb{Z}) \right)
    \]
    has exponent bounded by \( C \).
\end{theorem}

 
Recall that the \emph{exponent} of a group \( G \) is defined as the smallest positive integer \( N \) such that for every element \( \sigma \in G \), \( \sigma^N \) equals the identity element. In particular, Theorem \ref{main} implies that the local $b_1$ of the regular loci is zero, i.e., $$\mathrm{Im}\left( H_1(B_t(p) \cap \mathcal{R}, \mathbb{R}) \rightarrow H_1(B_1(p) \cap \mathcal{R}, \mathbb{R}) \right)=0.$$

 Theorem \ref{main thm-unrescaled version} is a consequence of the following result concerning the effective regular loci. 
\begin{theorem}\label{main}
    Let  $\epsilon=\epsilon(n)/2$, where $\epsilon(n)$ is the dimensional constant given in Theorem \ref{epsilon regularity}.  Given \( v \in (0, 1] \), \( n \in \mathbb{Z}_+ \), \( \delta > 0 \), there exist positive constants \( t(n, v) \), \( C(n, v) \) and \( \delta'(n, v, \delta)\leq \delta \) such that if \( (X, d, p) \in \mathcal{M}(n, v) \), then the group
\[
\mathrm{Im}\left( H_1\left( B_t(p) \cap \mathcal{R}_{\epsilon, \delta}, \mathbb{Z} \right) \rightarrow H_1\left( B_1(p) \cap \mathcal{R}_{\epsilon, \delta'}, \mathbb{Z} \right) \right)
\]
has exponent bounded by \( C \).
\end{theorem}

Theorem~\ref{main thm-unrescaled version} has several applications. First, we recall that in the proof of the uniform H\"ormander construction for polarized K\"ahler-Einstein metrics in ~\cite{DS12}, a key technical difficulty arises due to a topological obstruction stemming from the holonomy of flat connections on the complement of the singular set in a tangent cone. With Theorem~\ref{main thm-unrescaled version}, the topological obstruction can be ruled out \emph{a priori}. See also the discussion in Section~5.1 of~\cite{DS12}.

Secondly, Theorem \ref{main} yields stronger structural results for singularities of Ricci limit spaces in the K\"ahler setting. Let $\mathcal{C}$ be a metric cone which is a pointed Gromov-Hausdorff limit of non-collapsed \emph{K\"ahler} metrics with uniformly bounded Ricci curvature. It follows from \cite{DS15, LiuSz2} that $\mathcal{C}$ is naturally a normal affine variety. A corollary of Theorem \ref{main} is

\begin{corollary}\label{cor-pi1 of cone}
$\mathcal{C}$ has log terminal singularities.
\end{corollary} 

This result is independently proved by Hallgren-Sz\'ekelyhidi using different method \cite{HS}.
A consequence of Corollary \ref{cor-pi1 of cone} is that the 2-step degeneration theory for local singularities in \cite{DS15} can be extended from the case of K\"ahler-Einstein metrics to the case of K\"ahler metrics with bounded Ricci curvature. 
 Furthermore, one can generalize \cite[Proposition 4.15]{DS12} for  K\"ahler-Einstein metrics to the case of bounded Ricci curvature. Namely, 
let $Z$ be a pointed Gromov-Hausdorff limit of a sequence of non-collapsed \emph{polarized} K\"ahler manifolds with bounded Ricci curvature,  by \cite{DS12,DS15}, we know that $Z$ is a normal complex analytic space. Using Corollary \ref{cor-pi1 of cone} and H\"ormander's $L^2$-technique, we obtain that 
\begin{corollary}\label{cor--klt in general}
   $Z$ has log terminal singularities.
\end{corollary}

Lastly, in the Riemannian setting,  Theorem~\ref{main thm-unrescaled version}, together with the results from~\cite{Pan-Wei}, gives the following global statement.

\begin{corollary}\label{cor--compact version}
For any \( v, A > 0 \), there exists a constant \( C(n, v, A) \) with the following property. Let \( X \) be a Gromov–Hausdorff limit of closed Riemannian manifolds \( (M^n_i, g_i) \) satisfying
\[
\frac{1}{A} \leq \mathrm{Ric}_{g_i} \leq A \quad \text{and} \quad \mathrm{Vol}(M_i) \geq v.
\]
Then the group \( H_1(X^{\mathrm{reg}}, \mathbb{Z}) \) has exponent bounded by \( C \).
\end{corollary}

The main idea behind the proof of Theorem \ref{main} is to perform an induction on the volume ratio lower bound $v$. By Anderson’s $\epsilon$-regularity theorem, the statement holds when $v$ is close to 1. For a loop contained in the unit ball and lying in the regular part, we first show that, up to a bounded index, it is homotopically trivial within a ball of fixed large radius that contains the singular set. This can be seen as a local effective version of Anderson’s classical result \cite{And90} that a complete Riemannian manifold with non-negative Ricci curvature and Euclidean volume growth has a finite fundamental group. As a consequence, the given loop is homotopic to a sum of small loops within the regular part. Since we are passing to smaller balls, we expect a jump in the volume ratio—unless the space is close to a metric cone, by the “almost volume cone implies almost metric cone” theorem \cite{CC96}. The use of the effective regular set, along with an effective version of the loop decomposition theorem (see Lemma \ref{small loops}), allows us to carry out this idea rigorously via a contradiction argument. We remark that the reason we only obtain an exponent bound for the local homology group—rather than an exponent for the local fundamental group or a direct order bound on the local homology group—is due to the non-commutativity of the fundamental group and when passing to smaller balls, we lose uniform control over how many are needed.

In Section \ref{sec--preliminary}, we collect fundamental definitions and properties related to the effective regular set for non-collapsed Ricci limit spaces.
In Section \ref{proof of the main theorem}, we present the proofs of Theorem \ref{main thm-unrescaled version} and Theorem \ref{main}. We first establish Theorem \ref{weak main}. Then we prove a loop decomposition Lemma \ref{small loops}. Using a rescaling argument based on Lemma \ref{small loops}, we derive Theorem \ref{main} from Theorem \ref{weak main}. Theorem \ref{main} is an easy consequence of Theorem \ref{main thm-unrescaled version}.
In Section \ref{sec--applications}, we prove the corollaries.
In Section \ref{sec--discussion}, we list some problems for further study.

\subsection*{Acknowledgement} The authors thank Max Hallgren and G\'abor Sz\'ekelyhidi for sharing their draft.
The third author thanks  G\'abor Sz\'ekelyhidi and Valentino Tosatti for helpful discussions on this topic. Part of the work was done when the authors were at the SLMath (MSRI) during Fall 2024 (partially supported by the NSF Grant DMS-1928930).

\section{Preliminaries}\label{sec--preliminary}

In this section we collect some known results on $\mathcal M(n, v)$. First, we note that the set \(\mathcal{R}_{\epsilon} \) is scaling-invariant. Indeed, if we scale the distance on \( X \) by \( \lambda > 0 \), then an \( (\epsilon, \delta) \)-regular points becomes an \( (\epsilon, \lambda \delta) \)-regular point with respect to the rescaled metric.

The following two results depend crucially on the Ricci upper bound hypothesis.

\begin{theorem}[\cite{A90,CC97}]\label{epsilon regularity}
    There exists an $\epsilon(n)>0$ such that for $X\in \mathcal{M}(n,v)$ (for any $n\geq 1$ and $v>0$),  we have
    \begin{equation}
\mathcal{R}_{\epsilon(n)}=\mathcal{R},
    \end{equation}and if $p\in \mathcal R_{\epsilon,1}$, then $B_{1/2}(p)\subset \mathcal R$.
    Moreover, $\mathcal{R}$ is a smooth manifold endowed with a $C^{1, \alpha}$ Riemannian metric.
\end{theorem}

\begin{theorem}[\cite{CN15}] \label{CN}
For $(X, d, p) \in \mathcal{M}(n, v)$, every tangent cone at a point $q\in \mathcal S$ is of the form $\mathbb R^{n-k}\times C(Y)$ for $k\geq 4$, where $C(Y)$ does not split any line.
\end{theorem}

In the rest of this section we will fix $\epsilon=\frac{\epsilon(n)}{2}$, where $\epsilon(n)$ is given by Theorem \ref{epsilon regularity} and list some basic properties of the effective regular set.

\begin{lemma}\label{close regular points}
There exists $c(n)\leq 1$ such that 
\begin{itemize}
    \item [(1).] if $(M,p), (N,q)\in \mathcal{M}(n,v)$ for some $v>0$ and satisfy $d_{\mathrm{pGH}}((M,p),(N,q)) \le cr$, and $p \in \mathcal{R}_{\epsilon,r}$ for some $r\in (0,1]$, then $q \in \mathcal{R}_{\epsilon,cr}$;
    \item [(2).] if \( (M, p)\in \mathcal M(n,v) \) for some $v>0$ and \( p \in \mathcal{R}_{\epsilon, r} \) for some \( r \in (0, 1] \), then for any \( q \in B_{r/2}(p) \), we have \( q \in \mathcal{R}_{\epsilon, c r} \).
    
\end{itemize}
 
\end{lemma}

Here $d_{pGH}$ denote the pointed Gromov-Hausdorff metric on the space of isometry classes of pointed proper metric spaces \cite{Herron}, which induces the pointed Gromov-Hausdorff topology. The latter can also be defined using the $\epsilon$-Gromov-Hausdorff approximation \cite[Proposition 3.5]{Herron}.

\begin{proof}It suffices to prove the case \( r = 1 \). 

Assume that  the first statement does not hold. Then, we can find sequences \( (M_i, p_i) \) and \( (N_i, q_i) \) satisfying   
\[
d_{\text{pGH}}((M_i, p_i), (N_i, q_i)) \leq i^{-1},
\] and \( p_i \in \mathcal{R}_{\epsilon,1} \), but \( q_i \notin \mathcal{R}_{\epsilon, i^{-1}} \). 

Then, \( (M_i, p_i) \) and \( (N_i, q_i) \) converge to the same limit \( (X, p) \). Since $p_i\in \mathcal R_{\epsilon,1}$, we know that \( p \) is a regular point. Then we can find a sequence \( r_i \to \infty \) with  \( r_i / i \to 0 \), such that 
\[
(r_i N_i, q_i) \xrightarrow{\text{pGH}} \mathbb{R}^n.
\]
However, \( q_i \in r_i N_i \) is not in \( \mathcal{R}_{\epsilon, r_i/i} \). By the Reifenberg theorem \cite{CC97}, for sufficiently large \( i \), \( q_i \) must be in \( \mathcal{R}_{\epsilon, 1} \), which leads to a contradiction.

Similarly, we argue by contradiction to prove the second statement. Suppose there exists a sequence $(M_i,p_i)$ such that $p_i\in \mathcal{R}_{\epsilon,1}$ and $q_i\in B_{1/2}(p_i)$ such that $q_i\notin \mathcal{R}_{\epsilon,\frac{1}{i}}.$ Then passing to a subsequence, we can assume
\[
(M_i, p_i) \xrightarrow{\text{pGH}} (X,p).
\] Then $(X,p)\in \mathcal{M}(n,v)$ for some $v>0$ and $p\in \mathcal R_{\epsilon,1}$. By the choice of $\epsilon$, we know that there exists $\delta>0$ such that $B_{1/2}(p)\subset \mathcal{R}_{\epsilon,\delta}$, which implies a contradiction by the statement in (1).
\end{proof}

In the following we show that the effective regular set is effectively dense and almost geodesically convex.

\begin{lemma}\label{regular point dense}
Given $n\in \mathbb Z_+$ and $v>0$, there exists a constant \( \lambda(n, v) \in (0,1)\) such that if \( (X, p) \in \mathcal{M}(n, v) \), then \( \mathcal{R}_{\epsilon, \lambda} \cap B_1(p) \neq \emptyset \). In particular, for any $r\in (0,1]$,
$\mathcal{R}_{\epsilon,\lambda r} \cap B_r(p) \neq \emptyset$.

\end{lemma}

\begin{proof}
Assume that we can find $(X_i,q_i)$ satisfying the assumptions but with $\mathcal{R}_{\epsilon,i^{-1}} \cap B_1(q_i) = \emptyset$.
Passing to a sequence, we get a limit $(X,q)$ and by Lemma \ref{close regular points}, we would get that for any $r>0$, we have $B_{1/2}(q) \cap \mathcal{R}_{\epsilon,r} = \emptyset$, which contradicts with the fact that the regular point $\mathcal{R}$ has full $n$-dimensional Hausdorff measure.
\end{proof}


\begin{lemma}\label{regular point path connected}
Given $n\in \mathbb Z_+$, $v>0$ and $r\in (0, 1]$, there exists a constant $\mu(n,v,r)>0$ such that if $(X,p)\in \mathcal M(n,v)$, then for  any $q_1,q_2 \in B_{20s}(p) \cap \mathcal R_{\epsilon,rs}$ (for some $s\in (0, 1]$),there is a minimizing geodesic connecting $q_1$ and $q_2$ which is contained in $ \mathcal{R}_{\epsilon,\mu rs}$. 
\end{lemma}
\begin{proof}
By rescaling we may assume $s=1$.
 For a given \( r \in (0, 1] \), suppose there exists a sequence \( (X_i, p_i) \in \mathcal{M}(n, v) \) with \( q_{1,i}, q_{2,i} \in B_{20}(p_i) \cap \mathcal{R}_{\epsilon, r} \), but there is no minimizing geodesic from \( q_{1,i} \) to \( q_{2,i} \) in \( \mathcal{R}_{\epsilon, 1/i} \). Passing to a subsequence, we obtain that \( (X_i, p_i) \xrightarrow{\text{pGH}} (X, p) \), and let \( q_{1, \infty} \) and \( q_{2, \infty} \) denote the limits of the sequences \( q_{1,i} \) and \( q_{2,i} \), respectively. 

Then, we know that \( q_{1,\infty}, q_{2,\infty} \in \mathcal{R}_{\epsilon, r} \). Let $\gamma_i$ be a geodesic from $q_{1,i}$ to $q_{2,i}$. Passing to a subsequence if necessary, $\gamma_i$ point-wisely converges to a geodesic $\gamma$ from $p_{1,\infty}$ to $p_{2,\infty}$.  By Lemma \ref{close regular points} and the H\"older continuity estimate in \cite[Theorem 1.1]{CN12}, we know that there exists \( \delta > 0 \) such that $\gamma \subset \mathcal{R}_{\epsilon,\delta}$. Then by Lemma \ref{close regular points}, for $i$ large we know $\gamma_i \subset \mathcal{R}_{\epsilon,c\delta}$, a contradiction.

\end{proof}

For a metric cone $\mathcal C=C(Y)$ in $\mathcal M(n,v)$ for some $n\geq 3$ and $v>0$, we denote by $\mathrm{Vol}(\mathcal C)$  the ratio between the volume of unit ball centered at the vertex of the cone and the volume of the unit ball in $\mathbb R^n$. 

\begin{theorem}[\protect{\cite[Theorem 7.7]{Pan-Wei}}]\label{pi1 bounded of the link}
  We have
\[
\left| \pi_1(Y) \right| \leq \frac{1}{\mathrm{Vol}\left( C(Y) \right)}.
\]

\end{theorem}

\section{Proof of the main theorem}\label{proof of the main theorem}
We first prove an
estimate about the group order of local $H_1$ and then we show that Theorem \ref{main} follows from this estimate together with the loop decomposition result in Lemma \ref{small loops}.  It is important here that we work with the effective regular set, which is preserved under taking Gromov–Hausdorff limits and enables us to carry out a contradiction argument.
\begin{theorem}\label{weak main}
Fix $n\in \mathbb Z_+$. Let  $\epsilon=\epsilon(n)/2$, where $\epsilon(n)$ is the dimensional constant given in Theorem \ref{epsilon regularity}.    Given \( v \in (0,1] \) and  \( \delta \in (0,1) \), there exist positive constants \( t(n,v) \), \( C(n,v,\delta) \), and \( \delta'(n,v,\delta) \) such that if \( (X, d, p) \in \mathcal{M}(n, v) \), then
\[
\left| \mathrm{Im}\left( H_1\left( B_t(p) \cap \mathcal{R}_{\epsilon, t\delta}, \mathbb{Z} \right) \rightarrow H_1\left( B_1(p) \cap \mathcal{R}_{\epsilon, \delta'}, \mathbb{Z} \right) \right) \right| \leq C.
\]
\end{theorem}

To prove this, we make a few preparations.
First, we establish the following effective version of Anderson’s result \cite{And90}, which guarantees that a loop in a given geodesic ball is contractible, up to a finite index, within a larger geodesic ball of definite size.

\begin{lemma}\label{manifolds loop contractible}
Given $v > 0$ and $n\in \mathbb Z_{+}$, there exist constants $t_0(n,v)>0,L_0(n,v)>0$ depending only on $n$ and $v$ such that if a $n$-manifold satisfies $(M,p)$, $\mathrm{Ric} \ge -(n-1)$ and \begin{equation}
         \begin{aligned}
         \underline{\vol}(B_r(q))&:=\frac{\vol(B_r(q))}{\omega_nr^n} \geq v, \text{ for any $r\in (0,1]$ and $q\in B_1(p)$.}
    \end{aligned}
    \end{equation}
    Then 
        \begin{equation}
		\left|\mathrm{Im}(\pi_1(B_{2t_0}(p),p)\rightarrow \pi_1(B_{0.8}(p),p))\right|\leq  L_0.
	\end{equation}
\end{lemma}
\begin{proof}
Assume that the lemma does not hold, then there is a sequence $(M_i,d_i,p_i)$ such that $\mathrm{Ric} \ge -(n-1)$, \begin{equation}
         \begin{aligned}
         \underline{\vol}(B_r(q))&:=\frac{\vol(B_r(q))}{\omega_nr^n} \geq v, \text{ for any $r\in (0,1]$ and $q\in B_1(p_i)$.}
    \end{aligned}
    \end{equation}
and $G_i=\mathrm{Im}(\pi_1(B_{\frac{1}{i}}(p_i),p_i)\rightarrow \pi_1(B_{0.8}(p_i),p_i))$ satisfies $\left|G_i\right| >  i$.

Scaling up the metric \( d_i \) by a factor of \( i \), all geometric quantities in the following will be measured with respect to the rescaled metric. Define \( U_i = B_{0.8i}(p_i) \), and let \( (\widetilde{U}_i, \tilde{d}_i, \tilde{p}_i) \) denote the universal cover of \( U_i \) with the lifted metric \( \tilde{d}_i \) and base point \( \tilde{p}_i \).
There exists a constant \( c_1(n) > 0 \), depending only on the dimension \( n \), such that
\[
\mathrm{Vol}\left( B_{\frac{2i}{3}}(\tilde{p}_i) \right) \leq c_1(n)\, \omega_n i^n.
\]
We can construct a canonical fundamental domain of \( U_i \) in \( \widetilde{U}_i \), centered at \( \tilde{p}_i \). Let
\[
S_i = U_i \cap B_{i/2}(\tilde{p}_i),
\]
which satisfies
\[
\mathrm{Vol}(S_i) = \mathrm{Vol}\left( B_{i/2}(p_i) \right).
\]
Since \( (X_i, d_i, p_i) \in \mathcal{M}(n, v) \), by definition we have the volume lower bound:
\begin{equation}\label{eq--volume upper bound}
\mathrm{Vol}\left( B_{i/2}(p_i) \right) \geq v\, \omega_n i^n.
\end{equation}
Thus,
\[
\mathrm{Vol}(S_i) \geq v\, \omega_n i^n.
\]

Since $\pi_1(B_1(p_i),p_i)$ is generated by loops with length less than $3$, $G_i$ is generated by elements in 
$$G_i(\tilde{p}_i,3):=\{ g \in G_i\mid \tilde d_i(\tilde{p},g\tilde{p}) \le 3\}.$$

Fix a constant $C> \frac{c_1(n)}{v}$ and choose $i>20 C$. We claim that $\left|G_i(\tilde{p}_i,i/6)\right| > C$.
 If $G_i(\tilde{p}_i,i/6) = G_i$, then  $\left|G_i(\tilde{p}_i,i/6)\right| > i > 20C$ by the assumption $|G_i|\geq i$ and the choice of $i$. If $G_i(\tilde{p}_i,i/6) \neq G_i$, then this means that $G_i$ has an element moving $\tilde{p}_i$ out of $B_{i/6}(\tilde{p})$.
 Since any element in $G_i(\tilde{p}_i,3)$ can move $\tilde{p}$ at most by distance $3$ and $G_i$ is generated by elements in $G_i(\tilde{p}_i,3)$, the existence of an element in $G_i$ moving $\tilde{p}_i$ out of $B_{i/6}(\tilde{p})$ implies that there exists at least $\frac{i}{18}$ elements $g_1,\cdots, g_\frac{i}{18}$ in $G_i(\tilde{p}_i,3)$ such that for any $1\leq k\leq \frac{i}{18}$, the product
 \begin{equation}
     g_1\cdots g_k
 \end{equation}defining different elements in $G_i$
 So we have $\left|G_i(\tilde{p}_i,i/6)\right| \ge i/18 > C$. Therefore we always have $\left|G_i(\tilde{p}_i,i/6)\right| > C$.

 Then , we obtain the volume lower bound of the $G_i(\tilde{p}_i,i/6)$-orbit of $S_i$, 
 \begin{equation}
\vol\left(G_i(\tilde{p}_i,i/6)(S_i) \right)\geq C  v\omega_ni^n.
 \end{equation} On the other hand, $G_i(\tilde{p}_i,i/6)(S_i) \subset B_{2i/3}(\tilde{p}_i)$ and therefore by \eqref{eq--volume upper bound}, we have 
 \begin{equation}
     \vol\left(G_i(\tilde{p}_i,i/6)(S_i) \right)\leq c_1(n)\omega_n i^n,
 \end{equation}which contradicts to the choice of $C$.
\end{proof}

\begin{lemma}\label{contractible in big space}
		Given $v\in (0,1]$ and $n\in \mathbb Z_{+}$, take $t_0(n,v)>0,L_0(n,v)>0$ in Lemma \ref{manifolds loop contractible}. Then for any $(X,d,p)\in \mathcal{M}(n,v)$, we have
        \begin{equation}
		\left|\mathrm{Im}(\pi_1(B_{t_0}(p),p)\rightarrow \pi_1(B_{0.9}(p),p))\right|\leq  L_0.
	\end{equation}

\end{lemma}
\begin{proof}
Since $(X,d,p)\in \mathcal{M}(n,v)$, we can take a sequence of $(M_i,p_i) \xrightarrow{GH}(X,p)$ so that $\mathrm{Ric} \ge -(n-1)$ and \begin{equation}
         \begin{aligned}
         \underline{\vol}(B_r(q))&:=\frac{\vol(B_r(q))}{\omega_nr^n} \geq v, \text{ for any $r\in (0,1]$ and $q\in B_1(p)$.}
    \end{aligned}
    \end{equation}
By Lemma \ref{manifolds loop contractible}, we have for any $i$,
\begin{equation}
		\left|\mathrm{Im}(\pi_1(B_{2t_0}(p_i),p_i)\rightarrow \pi_1(B_{0.8}(p_i),p_i))\right|\leq  L_0.
	\end{equation}
By \cite[Lemma 4.2]{Wang2021} (see also the proof of Lemma \ref{contradiciton sequence}),  we know that for sufficiently large $i$, there is a natural homomorphism from 
$$\mathrm{Im}(\pi_1(B_{2t_0}(p_i),p_i)\rightarrow \pi_1(B_{0.8}(p_i),p_i))$$
to 
$$\mathrm{Im}(\pi_1(B_{3t_0}(p),p)\rightarrow \pi_1(B_{0.9}(p),p)).$$ Moreover, the image of this homomorphism contains
 $$\mathrm{Im}(\pi_1(B_{t_0}(p),p)\rightarrow \pi_1(B_{0.9}(p),p)).$$ Then we have 
 \begin{equation}
		\left|\mathrm{Im}(\pi_1(B_{t_0}(p),p)\rightarrow \pi_1(B_{0.9}(p),p))\right|\leq  L_0.
	\end{equation}
\end{proof}
\begin{remark}
	In the above lemma, it is sufficient to assume a Ricci lower bound. Since the first homology group is the abelianization of the fundamental group, we also obtain the following result: 
 \begin{equation}
		\left|\mathrm{Im}(H_1(B_{t_0}(p),\mathbb{Z})\rightarrow H_1(B_{0.9}(p),\mathbb{Z}))\right|\leq L_0.
 \end{equation}
\end{remark}

In our proof of Theorem \ref{weak main} we will argue by contradiction. The following Lemma allows us to pass from the limit space back to the converging sequence. 

\begin{lemma}\label{contradiciton sequence}
     Given $v>0$, $n\in \mathbb Z_+$ and a sequence $(X_i,p_i) \in \mathcal{M}(n,v)$ with $$(X_i,p_i) \xrightarrow{pGH}(X,p).$$ Let $c=c(n)$ be the constant appearing in Lemma \ref{close regular points} and $t_0=t_0(n,v)$ be the constant in Lemma \ref{contractible in big space}. Then for any $t\leq t_0$ and $\delta,\delta'\in (0,1)$ with $\delta'\leq \delta$ and sufficiently large $i$, there exists a natural homomorphism $h$ from the group 
\begin{equation*}
  \mathrm{Im}\left(  H_1(B_{1.1t}(p) \cap \mathcal{ R}_{\epsilon,\delta},\mathbb{Z}) \rightarrow H_1(B_{0.9}(p)\cap \mathcal {R}_{\epsilon,\delta'},\mathbb{Z})\right)
\end{equation*} to the group
\begin{equation*}
    \mathrm{Im}\left(H_1(B_{1.2 t}(p_i) \cap \mathcal{ R}_{\epsilon,c\delta},\mathbb{Z}) \rightarrow H_1(B_{1}(p_i)\cap \mathcal {R}_{\epsilon,c^2\delta'},\mathbb{Z})\right).
\end{equation*} Moreover, the image of $h$ contains the group
\begin{equation*}
    \mathrm{Im}\left(H_1(B_{t}(p_i) \cap \mathcal{ R}_{\epsilon,\delta/c},\mathbb{Z}) \rightarrow H_1(B_{1}(p_i)\cap \mathcal {R}_{\epsilon,c^2\delta'},\mathbb{Z})\right).
\end{equation*}

\end{lemma}
 \begin{proof} For any $\delta,\delta'\in (0,1)$ with $\delta \geq \delta'$, we first construct the claimed homomorphism $h$ for $i$ sufficiently large (depending on $\delta$ and $\delta'$). Fix Gromov-Hausdorff approximations which realize the pointed Gromov-Hausdorff convergence. For any loop $\gamma \subset B_{1.1t}(p) \cap \mathcal{ R}_{\epsilon,\delta}$, we can find $\gamma_i \subset B_{1.2t}(p_i)$ which is point-wise close to $\gamma$ (under the chosen Gromov-Hausdorff approximation). Then $\gamma_i \subset B_{1.2t}(p_i) \cap \mathcal{ R}_{\epsilon,c\delta}$ by Lemma \ref{close regular points}. Then we define that $h([\gamma])=[\gamma_i]$.

To show that $h$ is well-defined, we further assume that $\gamma$ is null-homologous in $B_{0.9}(p)\cap \mathcal {R}_{\epsilon,\delta'}$. We may assume that $\gamma$ is actually contractible so that we can use \cite[Lemma 2.4]{Pan-Wei}. Then for $i$ large $\gamma_i$ is homologous to a sum of small loops in $B_{1}(p_i)\cap \mathcal {R}_{\epsilon,c\delta'}$ where each loop is contained in a $c\delta'/3$-ball; see . Then by Reifenberg's Theorem, each small loop is contractible in a $c\delta'/2$-ball which is contained in $ \mathcal {R}_{\epsilon,c^2\delta'}$ due to Lemma \ref{close regular points}.  Thus $\gamma$ is null-homologous in $B_{1}(p_i)\cap \mathcal {R}_{\epsilon,c^2\delta'}$. By the same argument we also see that $h$ is a homomorphism.

For any $\gamma_i \subset B_{t}(p_i) \cap \mathcal{ R}_{\epsilon,\delta/c}$,  we can find $\gamma \subset B_1(p)$ point-wise close to  $\gamma$. By Lemma \ref{close regular points} and the above argument, we know $\gamma \subset B_{1.1t}(p) \cap \mathcal{ R}_{\epsilon,\delta}$. This proves the last statement.
 \end{proof}

To prove Theorem \ref{main} and Theorem \ref{weak main}, we rely on the following effective version of the loop decomposition theorem. It states that if a loop in the effective regular part is contractible in the whole space, then  it is homologous in the effective regular part to a sum of short loops whose length is comparable to its effective regular scale.

\begin{lemma}\label{small loops}
  Suppose $(X,d,p)\in \mathcal{M}(n,v)$. Let  $\lambda=\lambda(n,v)$ be given in Lemma \ref{regular point dense} and $\mu=\mu(n,v,\lambda/4)$ be given in Lemma \ref{regular point path connected}. Then for any $ r\in (0,1), \delta \le \lambda/16$, any loop $\gamma \subset B_{r}(p) \cap \mathcal{R}_{\epsilon,\delta}$ which is null-homologous in $B_{1}(p)$, we can find at most $C(n,v,\delta)$ many points $\{q_j\}$ in $B_{1}(p)$ so that $\gamma$ is homologous in $B_1(p)\cap \mathcal{R}_{\epsilon,\mu\delta}$  to the sum of loops, each of which is  contained in $$B_{\frac{16\delta}{\lambda}}(q_j) \cap \mathcal{R}_{\epsilon,\mu\delta}$$ for some $j$.
\end{lemma}
\begin{remark}\label{rescaling of decomposition}
If we rescale the metric by $\lambda/16\delta$ and let $B_1(q_j')$ denote the ball of radius one with respect to the rescaled metric, then  the set $B_{\frac{16\delta}{\lambda}}(q_j) \cap \mathcal{R}_{\epsilon,\mu\delta}$ becomes 
$$B_{1}(q_j') \cap \mathcal{R}_{\epsilon,\frac{\lambda\mu}{16}}.$$ Notice that $\lambda,\mu$ are independent of the choice of $\delta$. Such a rescaling argument would be critical in the proof of Theorem \ref{weak main} and \ref{main}
\end{remark}
\begin{proof}

Since $\delta/\lambda \le 1$, for any point $x \in B_{1}(p)$, $ B_{\delta/\lambda}(x)\cap \mathcal{R}_{\epsilon,\delta}  \neq \emptyset$ by Lemma \ref{regular point dense}. We can find at most $C(n,v,\delta)$ points $q_j$ in $B_{1}(p) \cap \mathcal{R}_{\epsilon,\delta}$ such  that the distance between each pair is at least $\delta/2\lambda$ and  for any $x \in B_1(p)$, there is some $q_j$ with $B_{3\delta/\lambda}(x) \subset B_{4\delta/\lambda}(q_j)$.

Given $\gamma$ as in the statement of the lemma, let $H\subset B_1(p)$ be a 2-chain with boundary given by $\gamma$. We can find a triangular  decomposition of $H$ so that the image of three vertices of each triangle is contained in a $\delta/\lambda$ ball and each vertex point is $\delta/\lambda$ close to a $\mathcal{R}_{\epsilon,\delta}$ point. We now modify $H$ to a new bounding chain $H'$ as follows. First we keep the triangular vertices of $H$ which are on $\gamma$. Then for those vertices not in $\gamma$, we may replace it by a point in $\mathcal R_{\epsilon, \delta}$ within distance $\delta/\lambda$. Then all the vertices are in $\mathcal{R}_{\epsilon,\delta}$ and the three vertices of each triangle are contained in a $3\delta/\lambda$-ball. In particular, any  two adjacent vertices are contained in some $B_{4\delta/\lambda}(q_j)$. 

We keep the image of $\gamma$ and next construct $H'$. By Lemma \ref{regular point path connected}, for any two adjacent vertices which are not both on $\gamma$, we can connect them by a  minimizing geodesic in $B_{8\delta/\lambda}(q_j) \cap \mathcal{R}_{\epsilon,\mu\delta}$.
Therefore the new triangle loop is contained in $B_{16\delta/\lambda}(q_j) \cap \mathcal{R}_{\epsilon,\mu\delta}$ for some $q_j$. By construction $\gamma$ is homologous to the sum of such loops. 
\end{proof}

\begin{proof}[Proof of Theorem \ref{weak main}]
 Consider the set 
\begin{equation}
    I:=\{v\in (0,1] \mid \text{Theorem } \ref{weak main} \text { holds for all }v'\geq v\}.
\end{equation}
First, we claim  $1\in I$. Indeed, by Colding's volume convergence theorem and Anderson's $\epsilon$-regularity theorem, if we  choose $t$ and $\delta'$ sufficiently small (depending only on $n$) then we have 
\begin{equation}
    B_{\sqrt t}(p)\subset \mathcal{R}_{\epsilon,\delta'}.
\end{equation}
From Lemma \ref{contractible in big space} and a standard rescaling argument, we see that Theorem \ref{weak main} holds for $v=1$. 

Now let $v_0=\inf I$. It suffices to prove $v_0=0$. Suppose $v_0>0$, then let $\lambda=\lambda(n,v_0/2),\mu=\mu(n,v_0/2,\lambda/4)$ be the constants given in Lemma \ref{close regular points} and let $L=2\max\{\frac{1}{t_0},L_0\}+1$, where $t_0=t_0(n,\frac{v_0}{2})$ and $L_0=L_0(n,\frac{v_0}{2})$ is the constant given by Lemma \ref{contractible in big space}. We set 
\begin{equation}
    \delta_0=\frac{\lambda\mu}{16}.
\end{equation} We first show that there exists $\tau\in (0,\frac{v_0}{2})$ such that the statement in Theorem \ref{weak main} holds for $v\in [v_0-\tau,v_0]$ and for this given $\delta_0$.

Suppose not then there exist a sequence  $(X_i,d_i,p_i) \in \mathcal{M}(n,v_0-\frac{1}{i})$ such that as $i\rightarrow \infty$,
\begin{equation}\label{eq-order goes to infinity}
	\left| \mathrm{Im}(H_1(B_{\frac{1}{i}}(p_i) \cap \mathcal{ R}_{\epsilon,\frac{\delta_0}{i}},\mathbb{Z}) \rightarrow H_1(B_1(p_i)\cap \mathcal {R}_{\epsilon,\frac{\delta_0}{i^2}},\mathbb{Z})) \right| \rightarrow \infty.
\end{equation}
Scaling up the metric $d_i$ by $i$ and passing to a subsequence, we may assume $$(X_i, i\cdot  d_i, p_i) \xrightarrow{pGH}(X,p).$$  Then $(X,p)\in \mathcal M(n,v_0/2)$ and for any $q\in X$ and $r>0$ we have 
$\underline{\vol}(B_r(q))\geq v_0.$
Without loss of generality we may assume $X\neq \mathbb R^n$.  We divide the discussion into two cases.

{\bf Case 1.}  $X$ is not a metric cone. 
Then we can find some $\sigma>0$ so that $\underline{\vol}(B_r(q)) \ge v_0+2\sigma$ for all $q \in B_{L+1}(p)$ and $r\in (0,1]$. Then by the volume convergence theorem \cite{Colding}, for $i$ sufficiently large, we have $(X_i',q') \in \mathcal{M}(n,v_0+\sigma)$ for any $q' \in B_{L}(p_i')$. Here we use $(X_i',p_i')$ to denote the metric space with respect to the rescaled metric $i\cdot d_i$ on $X_i$. Since we have assumed that Theorem \ref{weak main} holds for $v=v_0+\sigma$, there exist constants $t=t(n,v_0+\sigma)$, $\delta'=\delta(n,v_0+\sigma,\delta_0)$ and $C=C(n,v+\sigma,v_0)$ such that for any $q'\in B_{L}(p_i')$, we have  
\begin{equation}\label{bound on t balls}
    \left| \mathrm{Im}\left( H_1\left( B_{t}(q') \cap \mathcal{R}_{\epsilon, t\delta_0}, \mathbb{Z} \right) \rightarrow H_1\left( B_1(q') \cap \mathcal{R}_{\epsilon, \delta'}, \mathbb{Z} \right) \right) \right| <C.
\end{equation}
Then we want to show that the group 
\begin{equation}\label{uniform upper bound}
    \mathrm{Im}(H_1(B_1(p_i')\cap \mathcal R_{\epsilon,\delta_0},\mathbb Z)\rightarrow H_1(B_i(p_i')\cap \mathcal{R}_{\epsilon, \frac{\delta_0}{i}},\mathbb Z))
    \end{equation}has a uniform upper bound independent of $i$ and hence contradict with \eqref{eq-order goes to infinity}.
By Lemma \ref{contractible in big space}, in order to show that\eqref{uniform upper bound} has a uniform upper bound, it suffices to estimate the order the subgroups which consists of loops which are null-homologous in $H_1(B_{\frac{1}{t_0}}(p_i'),\mathbb Z)$. For such a loop $\gamma$, we can apply Lemma \ref{small loops} to get that $\gamma$ is homologous in $B_{\frac{1}{t_0}}(p_i')\cap \mathcal R_{\epsilon, t\delta_0}$ to the sum of loops, each of which is contained in 
\begin{equation}
    B_{t}(q_j')\cap \mathcal{R}_{\epsilon,t\delta_0}
\end{equation}for some $q_j'\in B_{\frac{1}{t_0}}(p_i')$. Then by \eqref{bound on t balls}, we obtain that \eqref{uniform upper bound} has a uniform upper bound independent of $i$.

{\bf Case 2.} $X$ is a metric cone. Then we may write $X= \mathbb{R}^{n-k} \times C(Y)$, where $C(Y)$ does not split any line.    By Theorem \ref{CN} we know $k\geq 4$.  If $B_{1.5L}(p)$ doesn't contain any cone vertex points $\mathbb{R}^{n-k} \times \{o_{C(Y)}\}$, then we can use the volume induction argument as in the Case 1. Therefore we may assume the distance between $p$ and the set $\mathbb{R}^{n-k} \times \{o_{C(Y)}\}$ is less $1.5L$.

 We denote by ${\bf r}$ the radial function on $C(Y)$. We identify $Y$ with the cross section in $C(Y)$ given by ${\bf r}=2L$. We also identify $C(Y)$ with $\{0^{n-k}\}\times C(Y)$. So $Y$ and $C(Y)$ are naturally viewed as subsets in $X$. 
 Then there exists a $\sigma>0$ such that for all $q\in Y$, $\underline{\mathrm{Vol}}(B_r(q))\geq v_0+\sigma$ for all $x\in B_1( q) \subset X$ and $r\in (0,1]$. 
Then since Theorem \ref{weak main} holds for $v_0+\sigma$, there exist constants $t=t(n,v_0+\sigma)$, $\delta'=\delta(n,v_0+\sigma,\delta_0)$ and $C=C(n,v+\sigma,\delta_0)$ such that for any $q\in Y$, we have  
\begin{equation}\label{case2 bound on t balls}
    \left| \mathrm{Im}\left( H_1\left( B_{t}(q) \cap \mathcal{R}_{\epsilon, t\delta_0}, \mathbb{Z} \right) \rightarrow H_1\left( B_1(q) \cap \mathcal{R}_{\epsilon, \delta'}, \mathbb{Z} \right) \right) \right| <C.
\end{equation} 

We only need to show that there exists $\eta>0$ so that 
\begin{equation}\label{need to prove in case 2}
    \mathrm{Im}(H_1(B_{1.1}(p) \cap \mathcal{ R}_{\epsilon,c\delta_0},\mathbb{Z}) \rightarrow H_1(B_{3L}(p)\cap \mathcal {R}_{\epsilon,\eta},\mathbb{Z}) < \infty,
\end{equation}
where $c$ is the constant in Lemma \ref{contradiciton sequence}. Then we will get a contradiction due to \ref{eq-order goes to infinity} and Lemma \ref{contradiciton sequence}.

We may assume that the cone point in $B_{1.5L}(p)$ is $(0^{n-k},o) \in X$.  Consider the natural projection map $\pi: X\setminus (\mathbb{R}^{n-k}\times \{o\}) \to Y$. Take  a small $\eta$ to be decided later.
Since cone vertices $\mathbb R^{n-k}\times \{o\}$ are all singular points, there is a natural homomorphism $h$ induced by $\pi$, from 
$$\mathrm{Im}(H_1(B_{1.1}(p) \cap \mathcal{ R}_{\epsilon,c\delta_0},\mathbb{Z}) \rightarrow H_1(B_{3L}(p)\cap \mathcal {R}_{\epsilon,\eta},\mathbb{Z})$$
to $H_1(Y,\mathbb{Z})$. By Theorem \ref{pi1 bounded of the link}, we have a uniform bound $|H_1(Y,\mathbb Z)| \leq v_0^{-1}$. Therefore we only need to show that $\ker (h)$ is bounded.

Take any $[\gamma] \in \ker(h)$ where $\gamma$ is a loop in  $B_{1.1}(p) \cap \mathcal{ R}_{\epsilon,c\delta}$ and the projection of $\gamma$ to $Y$ is contractible. Since the radial function on $B_{1.1}(p)$ is  strictly less than $2L$, $\pi(B_{1.1}(p)\cap \mathcal {R}_{\epsilon,c\delta_0}) \subset  Y \cap \mathcal{R}_{\epsilon,c\delta_0}$. Thus we may continuously deform $\gamma$ within $\mathcal R_{\epsilon, c\delta_0}$, so that it is contained in $Y$.   

Notice that the diameter of $Y$ has a uniform upper bound $4L\pi$. We take a small $t'<\min\{t(n,v_0+\sigma),c\}$ and then we take $\eta < t'\delta'/t$.
Applying Lemma \ref{small loops} to $X$, we see that there are $C(n,v_0,t')$ many points $\{q_j\}$ in $Y$ so that $\gamma$ is homologous in $\mathcal{R}_{\epsilon,\eta}$  to the sum of loops in $B_{t'}( q_j) \cap \mathcal{R}_{\epsilon,t' \delta_0}$. Rescaling the metric if necessary, we can apply \eqref{case2 bound on t balls} and conclude that for each $q_j \in Y$
\begin{equation}
    \left| \mathrm{Im}\left( H_1\left( B_{t'}(q_j) \cap \mathcal{R}_{\epsilon, t'\delta_0}, \mathbb{Z} \right) \rightarrow H_1\left( B_1(q_j) \cap \mathcal{R}_{\epsilon, \eta}, \mathbb{Z} \right) \right) \right| <C.
\end{equation} 
Then by a decomposition argument as in the last paragraph of the case 1 and using the fact that $B_1(q_j)\subset B_{3L}(p)$, we can achieve \eqref{need to prove in case 2}. The case 2 is proved.

%

\

Then we show that $[v_0-\tau,v_0]\in I$, that is Theorem \ref{weak main} holds for $[v_0-\tau,v_0]\in I$ and all $\delta>0$. Given $v\geq v_0-\tau\geq \frac{v_0}{2}$ and $(X,d,p)\in \mathcal{M}(n,v)$ and $\delta>0$. Without loss of generality, we may assume $\delta\leq \delta_0=\frac{\lambda\mu}{16}$. Then we claim that $t=t(n,v,\delta_0)$ will satisfy Theorem \ref{weak main}. Indeed consider a loop $\gamma \in B_t(p)\cap \mathcal{R}_{\epsilon,t\delta}$ . Then by Lemma \ref{small loops}, there are at most $C=C(n,v,\delta)$ many points $q_j\in B_1(p)$ such that $\gamma$ is homologous in $B_1(p)\cap \mathcal{R}_{\epsilon,\mu\delta}$ to the sum of loops each of which is contained in \begin{equation}
    B_{16t\delta/\lambda}(q_j)\cap \mathcal R_{\epsilon,t\mu\delta}
\end{equation} for some $q_j\in B_1(p)$. As mentioned in Remark \ref{rescaling of decomposition}, we rescale the metric by $\frac{\lambda}{16\delta}$, then with respect to the rescaled metric, it becomes $$ B_t(q_j')\cap \mathcal{R}_{\epsilon,\delta_0t}$$
  By the choice of $t$, Theorem \ref{weak main} holds for $v$ and $\delta_0$ and $t(n,v,\delta_0)$ and therefore with respect to the original metric, we know that there exists $\delta'=\delta'(n,v,\delta)$ and $C=C(n,v,\delta_0)$ 
\begin{equation}
    \left|\mathrm{Im}( H_1(B_{16t\delta/\lambda}(q_j)\cap \mathcal R_{\epsilon,t\mu\delta},\mathbb Z)\rightarrow H_1(B_{16\delta/\lambda}(q_j) \cap \mathcal{R}_{\epsilon,\delta'},\mathbb Z)\right|\leq C.
\end{equation}Then we can get an upper bound for 
\begin{equation}
\mathrm{Im}\left( H_1\left( B_t(p) \cap \mathcal{R}_{\epsilon, t\delta}, \mathbb{Z} \right) \rightarrow H_1\left( B_1(p) \cap \mathcal{R}_{\epsilon, \delta'}, \mathbb{Z} \right) \right).
\end{equation}

\end{proof}

\begin{proof}[Proof of Theorem \ref{main}]
Fix \( n \) and \( v \), and assume \( (X, d, p) \in \mathcal{M}(n, v) \). Let \( \lambda = \lambda(n, v) \) be the constant given in Lemma~\ref{regular point dense}, and let \( \mu = \mu(n, v, \lambda/4) \) be the constant from Lemma~\ref{regular point path connected}. Additionally, let \( t_0 = t_0(n, v) \) and \( L_0 = L_0(n, v) \) be the constants provided by Lemma~\ref{contractible in big space}. 
Furthermore, define  
\[
\tilde{t} = t(n, v), \quad 
\tilde{C} = C(n, v, \mu\lambda/16), \quad 
\tilde{\delta}' = \delta'(n, v, \mu\lambda/16)
\]  
as given by Theorem~\ref{weak main}.

Now, we choose \( \delta>0 \) sufficiently small such that
\[
\frac{16\delta}{\lambda \tilde{t}} < 0.1 \quad \text{and} \quad \delta < \frac{\lambda}{16}.
\]
For any \( t \leq t_0 \) and any loop \( \gamma \subset B_t(p) \cap \mathcal{R}_{\epsilon, \delta} \), by Lemma~\ref{contractible in big space}, we know that \( \gamma^{L_0!} \) is contractible in \( B_{0.9}(p) \).
Since \( \delta < \lambda/16 \), by Lemma~\ref{small loops} there exist finitely many points \( \{q_j\} \subset B_{1}(p) \) such that \( \gamma^{L_0!} \) is homotopic in \( \mathcal{R}_{\epsilon, \mu\delta} \cap B_1(p) \) to a sum of loops \( \gamma_j \), where for each \( j \), we have
\[
\gamma_j \subset B_{16\delta/\lambda}(q_j) \cap \mathcal{R}_{\epsilon, \mu\delta}.
\]
Next, we rescale the metric \( d \) on \( B_{16\delta/\lambda}(q_j) \cap \mathcal{R}_{\epsilon, \mu\delta} \) by a factor of \( \frac{\tilde{t}\lambda}{16\delta} \). Under this rescaled metric, each loop \( \gamma_j \) is contained within  
\[
B_{\tilde{t}}(q_j') \cap \mathcal{R}_{\epsilon,  {\tilde{t}}\mu\lambda/16}.
\]

By Theorem \ref{weak main} and our choice of \( \tilde{t}, \tilde{C} \), and \( \tilde{\delta}' \), we conclude that \( \gamma_j^{\tilde{C}!} \) is null-homologous in  
\[
B_1(q_j') \cap \mathcal{R}_{\epsilon, \tilde{\delta}'} 
\]
with respect to the rescaled metric.
Reverting to the original metric, we find that \( \gamma_j^{\tilde{C}!} \) is null-homologous in  
\[
B_{16\delta / (\lambda \tilde{t})}(q_j) \cap \mathcal{R}_{\epsilon, 16\tilde{\delta}'\delta / (\lambda \tilde{t})}.
\]
Then we can let
\[
\delta' = \delta'(n, v, \delta) = \min \left\{ \mu\delta, \frac{16\tilde \delta'\delta}{\lambda \tilde{t}} \right\}.
\]
and conclude that \( \gamma^{L_0! \tilde{C}!} \) is null-homologous in \( B_{1}(p) \cap \mathcal{R}_{\epsilon, \delta'} \).


\end{proof}

\begin{proof}[Proof of Theorem \ref{main thm-unrescaled version}]

We fix \( \epsilon = \frac{\epsilon(n)}{2} \) as before, where \( \epsilon(n) \) is the constant given in Theorem~\ref{epsilon regularity}. By definition, we have
\[
\mathcal{R} = \mathcal{R}_{\epsilon} = \bigcup_{\delta > 0} \mathcal{R}_{\epsilon, \delta}.
\]
Moreover, by~\cite[Lemma 2.5]{Pan2024}, we know that for any compact subset \( K \subset \mathcal{R} \), there exists some \( \delta > 0 \) such that
\[
K \subset \mathcal{R}_{\epsilon, \delta}.
\]
Since the image of any loop is compact, the statement follows from the fact that the constants \( t \) and \( C \) in Theorem~\ref{main} are independent of \( \delta \).

\end{proof}

\section{Applications}\label{sec--applications}

\begin{proof}[Proof of Corollary \ref{cor-pi1 of cone}
]
   By assumption, the regular set $\mathcal R$ of $\mathcal C$ is Ricci-flat K\"ahler, so on $\mathcal R$ the volume form $\omega^n$ defines a flat connection on the canonical bundle $K_{\mathcal R}$. By Theorem \ref{main thm-unrescaled version}, there is an integer $N$ such that $K_{\mathcal R}^{\otimes N}$ admits a parallel section $\Omega$ with $\|\Omega\|=1$. It follows that $$\int_{B\cap \mathcal R}(\Omega\wedge\bar\Omega)^{1/N}<\infty,$$ where $B$ is the unit ball around the vertex. By \cite[Lemma 6.4]{EGZ}, $\mathcal C$ has klt singularities. 
   
\end{proof}
\begin{remark}
    Combining with the results in \cite{Braun2021} and \cite{Xu-Zhuang2021}, we indeed have the explicit bound:
    $$|\pi_1(\mathcal R)|\leq \frac{1}{\vol(\mathcal C)}.$$
\end{remark}
\begin{remark}As mentioned earlier, it is known that 
 $\mathcal C$ is naturally a polarized normal affine variety \cite{DS15,LiuSz2}. Consequently homology groups of $\mathcal{C}^{reg}$ are finitely generated. Therefore Theorem \ref{main thm-unrescaled version} provides a proof, without relying on \cite{Braun2021}, that $H_1(\mathcal{C}^{reg},\mathbb Z)$ is finite. 
\end{remark}

\begin{proof}[Proof of Corollary \ref{cor--klt in general}]
    Fix a point $p\in Z$.
    By Corollary \ref{cor-pi1 of cone}, we know that there exists a positive integer $m$ such that for any tangent cone $\mathcal C$ at $p$, $K^{\otimes m}_{\mathcal{C}^{reg}}$ admits a parallel section. 
    
    For any $t>0$, we rescale the K\"ahler form on $Z$ by $t^{-2}$ and denote by $\omega_t$ the rescaled K\"ahler form and by $d_t$ denote the rescaled distance function to $p$. We still use $B_r(p_t)$ to denote the  ball of radius $r$ centered at $p$ with respect to the rescaled metric $d_t$. By assumption there exist a sequence of unit balls $B_1(p_{t,i})$ contained in a smooth polarized K\"ahler manifold $(Z_{t,i}, \omega_{t,i},d_{t,i},p_{t,i})$ such that 
    \begin{equation}\label{eq-smooth approximation}
        d_{\mathrm{pGH}}(B_1(p_{t,i}),B_1(p_t))=\Psi(i^{-1};t),
    \end{equation}where $\Psi(a;b)$ denotes a nonnegative function satisfying that for any given parameter $b$, $\lim_{a\rightarrow 0}\Psi(a;b)=0$. In the following, when we say \( i \) is sufficiently large, we always allow this to depend on \( t \). In the end, we will first choose \( t \) to be small and then let \( i \to \infty \).
    
    By Lemma \ref{regular point dense}, there exists $\lambda>0$ such that for any $\delta, t\in (0,1)$ and $i$ sufficiently large, we can find a point
    \begin{equation}
       q_{t,i}\in  B_{\delta}(p_{t,i})\cap \mathcal R_{\epsilon, \lambda \delta}.
    \end{equation}
    Using the argument in \cite[Section 3.2]{DS12}, we obtain that for $t$ sufficiently small and $i$ sufficiently large, we have on $ B_{0.9}(p_{t,i})$, the line bundle $K_{Z_{t,i}}^{\otimes m}$ admits a smooth section $\sigma_{t,i}$ such that 
    \begin{equation} 
    \begin{aligned}
         &\int_{B_{1/2}(p_{t,i})}|\sigma_{t,i}|^2\omega_{t,i}^n\geq 1, \int_{B_{0.9}(p_{t,i})}|\sigma_{t,i}|^2\omega_{t,i}^n\leq 10 \\
           & \int_{B_{0.9}(p_{t,i})}|\pp \sigma_{t,i}|^2\omega_{t,i}^n=\Psi(t;\delta)+\Psi(i^{-1};t,\delta),  \quad |\sigma_{t,i}(q_{t,i})|\geq 1.
    \end{aligned}
    \end{equation} By \cite[Proposition 3.1]{LiuSz1}, we know that for $t$ sufficiently small and $i$ sufficiently large, we can find a K\"ahler potential $\varphi_{t,i}$ on $B_1(p_t')$ such that 
    \begin{equation}\label{estimate on the potential}
      \omega_{t,i}=\ii\partial\pp \varphi_{t,i} \text{ and }  |\varphi_{t,i}-\frac{1}{2}d_{t,i}^2|=\Psi(i^{-1};t)+\Psi(t).
    \end{equation}Since the sublevel set of a strict smooth plurisubharmonic function has a pseudoconvex boundary, \eqref{estimate on the potential} ensures that we can apply the H\"ormander $L^2$-estimate on an open set $U_{t,i}$ with $B_{1/2}(p_{t,i})\subset U_{t,i}\subset B_1(p_{t,i})$. Then for $t$ sufficiently small and $i$ sufficiently large, we can solve $\pp \tau_{t,i}=\pp \sigma_{t,i}$ on $B_{1/2}(p_{t,i})$ with the estimate $$\int_{B_{1/2}(p_{t,i})} |\tau_{t,i}|^2\omega_{t,i}^n\leq \frac{1}{10} \int_{B_{0.9}(p_{t,i})}|\pp \sigma_{t,i}|^2\omega_t^n=\Psi(t;\delta)+\Psi(i^{-1};t,\delta).$$
By the choice of $q_{t,i}$ and standard elliptic estimates, we get 
\begin{equation}
    |\tau_{t,i}|(q_{t,i})\leq {\Psi(t;\delta,\lambda)}+\Psi(i^{-1};t,\delta,\lambda).
\end{equation}
Let $s_{t,i}=\sigma_{t,i}-\tau_{t,i}$, then we obtain that $$|s_{t,i}|(q_{t,i})=1-\Psi(t;\delta)-\Psi(i^{-1};t,\delta).$$ Moreover we have $\pp s_{t,i}=0$ and $\int_{B_{1/2}(p_{t,i})}|s_{t,i}|^2\omega_{t,i}^n\leq 11$.
Therefore  by the standard estimates for holomorphic sections (see  \cite[Section 2]{DS12}), there exists a constant $C$ independent of $\delta,$ $t$ or $i$, such that for $t$ sufficiently small and  $i$ sufficiently large
\begin{equation}\label{gradient bound}
    \|s_{t,i}\|_{L^\infty(B_{1/4}(p_t'))}+\|\nabla s_{t,i}\|_{L^\infty(B_{1/4}(p_{t,i}))}\leq C.
\end{equation}Now we choose $\delta$ such that $C\delta \leq 1/4$ and for this fixed $\delta$ we first choose $t$ sufficiently small and then $i$ sufficiently large such that we have 
\begin{equation}
    |s_{t,i}|(q_{t,i})\geq 1/2.
\end{equation}Then by \eqref{gradient bound}, we know that the holomophic section $s_{t,i}$ satisifies 
\begin{equation}\label{eq--bounds befores limit}
   |s_{t,i}|\geq \frac{1}{4} \text{ on $B_{1/4}(p_{t,i})$ } \quad \text{and} \quad  \int_{B_{1/2}(p_{t,i})}(s_{t,i}\wedge \bar s_{t,i})^{\frac{1}{m}}<11
\end{equation}Since by \eqref{eq-smooth approximation}, the ball $B_1(p_{t,i})$ converges to $B_t$ as $i\rightarrow \infty$ and due to the Ricci two sided bound, the metric tensor converges in the $C^{1,\alpha}_{loc}$ sense. Therefore by the estimate in \eqref{eq--bounds befores limit} and let $i\rightarrow \infty$, we get for $t$ sufficiently small, there exists a holomorphic section $s_t$ of $K_{Z^{reg}}^{\otimes m}$ on $B_{1/2}(p_t)\cap Z^{reg}$ such that $s_t$ is nowhere vanishing on $B_{1/8}(p_t)\cap Z^{reg}$ and 
\begin{equation*}
    \int_{B_{1/4}(p_{t})}(s_{t}\wedge \bar s_{t})^{\frac{1}{m}}<11.
\end{equation*}
Therefore 
$Z$ has log terminal singularities near $p$ by \cite[Lemma 6.4]{EGZ}. Since $p$ is arbitrary, we know that $Z$ has log terminal singularities.
    
\end{proof}

\begin{proof}[Proof of Corollary \ref{cor--compact version}]

Let $X$ be given in the Corollary \ref{cor--compact version}. Then by the Bonnet-Myers theorem and \cite[Corollary 7.1]{Pan-Wei}, we know there exists a constant $C_0=C_0(n,v,A)$ such that 
\begin{equation}
    |\pi_1(X)|\leq C_0.
\end{equation} Let $\iota:\mathcal R\hookrightarrow X$ denote the standard embedding and it induces a group homomorphism $$\iota_*: \pi_1(\mathcal R)\rightarrow \pi_1(X).$$ In order to estimate the exponent of the group $H_1(\mathcal R, \mathbb Z)$, it suffices to estimate order of the image of  elements in $\ker(\iota_*)$ under the natural map $\pi_1(\mathcal R)\rightarrow H_1(\mathcal R,\mathbb Z)$. Suppose $\gamma\in \ker(\iota_*)$. Since $\gamma$ is contractible in $X$, applying Lemma \ref{small loops}, we know that for any small positive $s$ there exist finitely many balls $B_{s}(p_j)$, such that $\gamma$ is homologous to a sum of loops, each of which is contained in $B_{s}(p_j)$ for some $j$.  By Theorem \ref{main thm-unrescaled version}, we know that there exists $t=t(n,v,A)$ and $L=L(n,v,A)$ such that for any $p\in X$, the group \begin{equation}\label{eq-local finiteness}
	\mathrm{Im}(H_1(B_t(p)\cap \mathcal R,\mathbb Z)\rightarrow H_1(\mathcal R,\mathbb Z))
\end{equation}has a uniform exponent bound. Then the results follows from the fact that the relative first homology group \eqref{eq-local finiteness} has a uniform exponent bound by choosing $s=t(n,v,A).$
\end{proof}

 \begin{remark}

We remark that both the positive lower bound and  the upper bounds on positive Ricci curvature are necessary. By the standard Kummer construction (see, for example, \cite{Donaldson2010}), it is known that there exists a sequence of non-collapsing Ricci-flat Kähler metrics converging to \( T^4/\mathbb{Z}_2 \), and we have  
\begin{equation}   H^1((T^4/\mathbb{Z}_2)^{\text{reg}}, \mathbb{R}) \neq 0.  
\end{equation}  
Furthermore, a topological 2-dimensional sphere with finitely many cone points can be obtained as the Gromov–Hausdorff limit of a sequence of non-collapsing metrics with sectional curvature bounded below by a positive number, for example, by rounding out the cone points.
 \end{remark}

\section{Discussion}\label{sec--discussion}
It is a folklore question whether a non-collapsed Ricci bounded (for example, Einstein) limit space has a real analytic structure. This is known if the dimension is at most 4 \cite{CN15} or if we work in the polarized K\"ahler setting \cite{DS12}.  Related to this, it is interesting to explore further the local topology of the regular loci in a non-collapsed bounded Ricci limit space. Theorem \ref{main thm-unrescaled version} only represents a first step in this direction. One expects a much stronger result to hold
\begin{conjecture}[Finite local fundamental group]\label{conj1}
    Given \( v \in (0,1] \) and \( n \in \mathbb{Z}_+ \), there exist constants \( t= t(n, v) > 0 \) and $C=C(n,v)>0$ such that for any \( (X, d, p) \in \mathcal{M}(n, v) \), we have
    \[
   \left| \mathrm{Im}\left( \pi_1(B_t(p) \cap \mathcal{R}) \rightarrow \pi_1(B_1(p) \cap \mathcal{R}) \right)\right|\leq C.
    \]
\end{conjecture}
Furthermore, for a given $(X, d, p)\in \mathcal M(n, v)$,  for $t>0$ small, one expects that the group $$\pi_1^{loc}(X, p):=
    \mathrm{Im}\left( \pi_1(B_t(p) \cap \mathcal{R}) \rightarrow \pi_1(B_1(p) \cap \mathcal{R}) \right)
   $$ is independent of $t$, and  its order is bounded by $\frac{1}{\vol_p}$, where   the \emph{volume density}  $\vol_p$ is defined to be  $\vol(\mathcal C)$ for any tangent cone $\mathcal C$ at $p$.  In this case, we can define this group to be the \emph{local fundamental group} at $p$, which is a local invariant of $X$.

In another direction, even without an upper on the Ricci curvature, one expects that a version of Theorem \ref{main thm-unrescaled version} and Conjecture \ref{conj1} to hold. In this context, the statement needs to be reformulated; in particularly it is necessary to enlarge $\mathcal 
R$ by including the codimension-2 singular set. A special situation of interest arises again in K\"ahler geometry. Let $\mathcal{P}(n,v)$ denote the set of Gromov-Hausdorff limits of pointed complete polarized K\"ahler manifolds of complex dimension $n$ satisfying $\mathrm{Ric}(g)\geq -1$ and   
\begin{equation}
   \underline{\vol} (B_r(q)):=\frac{\vol(B_r(q))}{\omega_nr^n} \geq v, \text{ for any $r\in (0,1]$ and $q\in B_1(p)$.}
\end{equation} Then by \cite{DS15,LiuSz1}, we know that any element $Z$ in $\mathcal{P}(n,v)$ admits a normal complex analytic variety structure. In this case the complex analytic regular part  $\mathscr{R}$ of $Z$ may be strictly larger than the metric regular part $\mathcal R$. A natural generalization of Theorem \ref{main thm-unrescaled version} to  this setting is the following.
\begin{conjecture}\label{Ricci lower bound}
     Given \( v \in (0,1] \) and \( n \in \mathbb{Z}_+ \), there exist constants \( t= t(n, v) > 0 \) and \( C = C(n, v) > 0 \) such that for any \( (Z, d, p) \in \mathcal{P}(n, v) \), the image of the homomorphism
    \[
    |\mathrm{Im}\left( \pi_1(B_t(p) \cap \mathscr{R}, \mathbb{Z}) \rightarrow \pi_1(B_1(p) \cap \mathscr{R}, \mathbb{Z}) \right)|\leq C.
    \]
\end{conjecture}
 It is likely that the proof presented in this paper, together with the ``complex-analytic $\epsilon$-regularity theorem" proved in \cite[Proposition 3.2]{LiuSz1}, can provide a proof of a weak version of the conjecture (similar to Theorem \ref{main}).
A positive answer to Conjecture \ref{Ricci lower bound}, or to some of its weaker versions, would shed light on the following question.

 \begin{question}
     Let $Z\in \mathcal{P}(n,v)$. Are the following two statements equivalent?
     \begin{itemize}
         \item The Ricci form is locally $\partial\pp$-exact; \item $Z$ locally has log terminal singularities.
     \end{itemize}
 \end{question}
 Here we say the Ricci form is locally $\partial\pp$-exact if for \textit{every} point $x\in Z$, locally there exist a  function $\varphi$, which is the sum of a plurisubharmonic function and a bounded function, such that as a current on the complex analytic regular part of $Z$ around $x$, we have 
 \begin{equation}\label{ric locally exact}
\mathrm{Ric}=\ii\partial\pp \varphi,
 \end{equation} where we recall that $\mathrm{Ric}$ is a well-defined closed (1,1)-current on the complex analytic regular part of $Z$ by \cite{LiuSz1}. 
 Note that one direction in the above question should be straightforward and can be justified as follows. Suppose $Z$ has locally log terminal singularity (indeed locally $\mathbb Q$-Gorenstein is enough), then we want to show that the Ricci form is locally $\partial\pp$-exact. To see this, for any given $x\in Z$ and let $s$ be a local nowhere vanishing holomorphic section of $K_{\mathscr{R}}^{\otimes N}$ for some $N\in \mathbb N$, where $\mathscr R$ denotes the complex analytic regular part of $Z$. By \cite[Proposition 3.1]{LiuSz1}, the singular K\"ahler metric $\omega_Z$ admits a local bounded  potential $u$ and on the complex analytic regular part near $x$, we have a plurisubharmonic function 
 \begin{equation}
f=\log\left(\frac{(s\wedge\bar s)^{\frac{1}{N}}}{\omega_Z^n}\right)+u,
 \end{equation}such that as a closed (1,1)-current on $\mathscr R$, we have 
 \begin{equation}
\mathrm{Ric}=\ii\partial\pp (f-u).
 \end{equation} We refer to \cite[Section 4]{LiuSz1} for the precise definition of $f$, and note that the key difference here is that, by assuming $Z$ is locally $\mathbb{Q}$-Gorenstein, we get a  well defined function $f$ on the complex analytic regular part near \emph{every} point, not just around a  complex analytic regular point as in \cite[Section 4]{LiuSz1}.
Since $\mathscr R$ has complex codimension at least 2 in $Z$, so $f$ extends to a plurisubharmonic function defined on a neighborhood of $x$ in $Z$  \cite{GR1956}. Then we get that \eqref{ric locally exact} holds locally near $x$ for some function $\varphi$, which is the sum of a plurisubharmonic function and a bounded function.

 Similar questions can be asked generally of non-collapsed RCD spaces induced by singular K\"ahler metrics on normal complex analytic varieties. See \cite{HS} for related discussions. 
   We leave these for future study.

\bibliographystyle{plain}
\bibliography{ref.bib}

\end{document}